\documentclass[12pt,a4paper,oneside]{amsart}

\usepackage{amsfonts, amsmath, amssymb, amsthm, amscd}
\usepackage{hyperref}
\hypersetup{
  colorlinks   = true, 
  urlcolor     = blue, 
  linkcolor    = blue, 
  citecolor   = red 
}
\usepackage{anysize}
\marginsize{2cm}{2cm}{1.5cm}{1.5cm}

\usepackage[pdftex]{graphicx}

\newtheorem{theorem}{Theorem}[section]
\newtheorem{lemma}[theorem]{Lemma}

\theoremstyle{definition}
\newtheorem{definition}[theorem]{Definition}

\theoremstyle{remark}
\newtheorem{remark}[theorem]{Remark}

\numberwithin{equation}{section}

\DeclareMathOperator{\diam}{diam}
\DeclareMathOperator{\conv}{conv}

\newcommand{\inte}{\mathop{\rm int}}

\renewcommand{\epsilon}{\varepsilon}
\renewcommand{\phi}{\varphi}
\renewcommand{\kappa}{\varkappa}

\newcommand{\gdiv}{\,\frac{\, *\, }{}\,}

\title{On the Carath\'eodory number for strong convexity}

\author{Vuong Bui{$^\spadesuit$}}

\author{Roman~Karasev{$^\clubsuit$}}

\thanks{{$^\clubsuit$} Supported by the Federal professorship program grant 1.456.2016/1.4 and the Russian Foundation for Basic Research grants 18-01-00036 and 19-01-00169}

\address{{$^\spadesuit$}Moscow Institute of Physics and Technology, Institutskiy per. 9, Dolgoprudny, Russia 141700 and Institut f\"ur Informatik, Freie Universit\"{a}t Berlin, Takustra{\ss}e~9, 14195 Berlin, Germany}
\email{bui.vuong@yandex.ru}

\address{{$^\clubsuit$}Moscow Institute of Physics and Technology, Institutskiy per. 9, Dolgoprudny, Russia 141700 and Institute for Information Transmission Problems RAS, Bolshoy Karetny per. 19, Moscow, Russia 127994}

\email{r\_n\_karasev@mail.ru}
\urladdr{http://www.rkarasev.ru/en/}


\begin{document}

\begin{abstract}
We give an improvement of the Carath\'eodory theorem for strong convexity (ball convexity) in $\mathbb R^n$, reducing the Carath\'eodory number to $n$ in several cases; and show that the Carath\'eodory number cannot be smaller than $n$ for an arbitrary gauge body $K$. We also give an improved topological criterion for one convex body to be a Minkowski summand of another.
\end{abstract}

\maketitle

\section{Introduction}

In the works \cite{pol1996,balashov2000} (recent references in English are \cite{holmsenkar2017,jahn2017}) a strengthening of the notion of convexity was studied. The first natural example is to call a convex body $A\subset\mathbb R^n$ \emph{$R$-strongly convex} for a positive real $R$, if $A$ is an intersection of balls of radius $R$. This notion seems to have been rediscovered many times (see the references in the cited works), and in \cite{pol1996,balashov2000} it was generalized to the following: For a fixed convex body $K\subset\mathbb R^n$, another convex body $A$ is \emph{$K$-strongly convex} if it is an intersection of translates of $K$.

The motivations to study the strong convexity may be different, the papers \cite{pol1996,balashov2000} were mostly motivated by certain optimization properties of strongly convex bodies or strongly convex functions, which we do not define here. But it turned out that this notion is mathematically interesting on its own. For example, $K$-strong convexity inherits good properties from ordinary convexity and $R$-strong convexity (corresponding to the case when $K$ is a Euclidean ball) under a certain ``generating set'' assumption on $K$. Note also that what we call here ``strong convexity'' was also studied in \cite{langi2013,jahn2017} under the name ``ball convexity''.

Let us give precise definitions and recall some notations from the cited works. We restrict ourselves to the case of a finite dimensional $\mathbb R^n$, this is the situation when one may expect Carath\'eodory-type theorems that we are going to discuss here.

\begin{definition}
\label{definition:genset}
A convex body $K\subset\mathbb R^n$ is a \emph{generating set} if any non-empty intersection of its translates
$$
K\gdiv T := \bigcap_{t\in T} (K - t)
$$
is a Minkowski summand of $K$, that is, $(K\gdiv T) + T' = K$ for some convex compactum $T'$.
\end{definition}

The simplest case of a generating set is a Euclidean ball in $\mathbb R^n$. In \cite{kar2001gen} it was shown that it is sufficient to test the generating set property for all two-point sets $T$; in \cite{ivanov2007} a nontrivial reformulation of this criterion was established and used to show that sufficiently small, in $C^2$ metric, deformations of Euclidean balls are generating sets. We will not use those nontrivial observations here, but an interested reader may consult the provided references. We only use the relatively simple facts (see \cite{balashov2000}) that in dimensions at most $2$ all convex bodies are generating sets and any Cartesian product of generating sets is a generating set.

\begin{definition}
\label{definition:strongconv}
Let $K\subset \mathbb R^n$ be a convex body. A set $C\subset \mathbb R^n$ is called {\em $K$-strongly convex} if it is an intersection of translates of $K$, that is 
$$
C = K\gdiv T
$$
in the above notation. 
\end{definition}

\begin{definition}
\label{definition:strong-convex-hull}
The minimal $K$-strongly convex set containing a given set $X$ is called its \emph{$K$-strongly convex hull}. If $X$ is contained in a translate of $K$ then it can be expressed by the following formula
$$
\conv_K X = K\gdiv (K\gdiv X).
$$
In case $X$ is not contained in any translate of $K$, we consider the $K$-strongly convex hull of $X$ undefined.
\end{definition}

In \cite{kar2001car} (as explained in \cite{holmsenkar2017} in English) it was shown that if $K$ is a generating set, then a version of the Carath\'eodory theorem (\cite{car1911}, see also the survey \cite{eckhoff1993}) for strong convexity holds, that is, for finite point sets $X\subset\mathbb R^n$ contained in a translate of $K$ we have
\[
\conv_K (X) = \bigcup_{Y\subseteq X,\ |Y|\le n+1} \conv_K (Y).
\]
In other words, under the generating set assumption we have the Carath\'eodory theorem for strong convexity with the same constant $n+1$ as for the usual convexity. More generally, without any assumption on $K$, we make the following definition:

\begin{definition}
We call $m$ the \emph{Carath\'eodory number} for $K$-strong convexity if this is the minimal $m$ such that  
\[
\conv_K (X) = \bigcup_{Y\subseteq X,\ |Y|\le m} \conv_K (Y)
\]
for any finite point set $X$ contained in a translate of $K$. 
\end{definition}

We sometimes speak about the Carath\'eodory number of a particular situation, where $K$, $X$, and a point $x$ are given and we want to conclude that
\[
x\in \bigcup_{Y\subseteq X,\ |Y|\le m} \conv_K (Y)\Rightarrow x\in \conv_K X,
\]
with the smallest possible $m$.

We show that the Carath\'eodory number can be reduced to $n$ in some particular cases. Assuming the generating set property of $K$, the Carath\'eodory number can be reduced to $n$ in the situation when the point $x$ lies outside $\conv X$ (the ordinary convex hull). This may be practically useful because the ordinary convex hull is always contained in the strong convex hull and the problem of determining the $K$-strongly convex hull is only interesting outside the ordinary convex hull. The precise statement of this results is:

\begin{theorem}
\label{theorem:caratheory-outside}
For a generating set $K\subset\mathbb R^n$ and $X\subset\mathbb R^n$, which has a translate in $K$, we have
\[
\conv_K (X)\setminus \conv X \subseteq \bigcup_{Y\subseteq X,\ |Y|\le n} \conv_K (Y).
\]
\end{theorem}

We also show that the Carath\'eodory number is $n$ (this time independently of the point $x$) for specific gauge bodies $K$.

\begin{theorem}
\label{theorem:product-gen}
If $L\subset\mathbb R^\ell$ and $M\subset \mathbb R^m$ are generating convex bodies and $K=L\times M\subset\mathbb R^n$, for $n=\ell+m$ with $\ell,m\ge 1$, then the Carath\'eodory number for $K$-strong convexity is precisely $n$. 
\end{theorem}

A good example for this product theorem is a nontrivial product $K = L_1\times \dots \times L_k$ of convex bodies $L_i$ of dimension at most $2$ each. Since all convex bodies in dimensions at most $2$ are generating sets, as well as their products, this theorem is applicable and the Carath\'eodory number of $K$-strong convexity equals $\dim K$ in this case. A more particular case of this is a parallelotope, which is a product of segments up to an affine transformation.

The proof of Theorem \ref{theorem:caratheory-outside} starts from choosing a point in $\conv_K (X)\setminus \conv X $. Suppose we translate the points and sets so that the point in question becomes $0\in \conv_K X\setminus \conv X$, then the theorem can be restated in the following alternative version. This statement has an advantage that it does not use the complicated definitions of $K$-strong convexity and is itself a nice theorem about coverings:

\begin{theorem}[Alternative statement of Theorem \ref{theorem:caratheory-outside}]
\label{theorem:helly-type}
Assume that $K\subset\mathbb R^n$ is a generating set, $X$ is a finite point set, $X$ can be covered by a translate of $K$, $0\not\in\conv X$, and any $\le n$ points of $X$ can be covered by a translate of $K$ not containing the origin $0$. Then $X$ can be covered by a translate of $K$ not containing the origin. 
\end{theorem}

It also makes sense to ask how small the Carath\'eodory number for $K$-strong convexity can be for some particular $K$, can it be smaller than $n$, the dimension of $K$. For example, when $K$ is a parallelotope (affine image of a cube) then the Helly number of the system of its translates is $2$ (see \cite[Chapter~III]{boltyanski1997} for different facts about the Helly number of a system of translates). The definition of the Helly number of a system of translates is very close to the statement of Theorem \ref{theorem:helly-type} above, the formula in the beginning of its proof in Section \ref{section:helly-type} relates the Carath\'eodory number for $K$-strong convexity to the Helly number of a family of translates of $K$ minus $K$ itself. Hence we might expect a low Carath\'eodory number for $K$-strong convexity when $K$ is a parallelotope. But in fact this number is never smaller than $n$, since we establish:

\begin{theorem}
\label{theorem:at-least-n}
If $K\subset\mathbb R^n$ is a convex body $($not necessarily a generating set$)$ then the Carath\'eodory number for $K$-strong convexity is at least $n$.
\end{theorem}

As for the upper bounds, the examples in \cite[Theorem 8]{langi2013} and \cite[Section 5]{holmsenkar2017} show that in dimensions $n\ge 3$ there is no upper bound for the Carath\'eodory number for strong convexity, while in dimensions $n\le 2$ all convex bodies are generating sets and the Carath\'eodory number is at most $n+1 \le 3$. In Section \ref{section:infinite} we give another example of an unbounded Carath\'eodory number, which in our opinion has a simpler geometric presentation.

Once we have started to consider the case when $K$ is not necessarily a generating set, we mention a fact about the Carath\'eodory number of products. This is a particular case of the main result of \cite{sierksma1975} about the Carath\'eodory number of a product of abstract convexity spaces, which in our setting has the following form:

\begin{theorem}[Sierksma, 1975]
\label{theorem:product-any}
Let $K$ be the Cartesian product $L \times M$ of two convex bodies $L$ and $M$. If $k, \ell, m$ are the strong Carath\'eodory numbers for $K, L, M$ respectively, then $\ell + m - 2 \le k \le \ell + m$.
\end{theorem}

Below we give proofs of the theorems stated here. In Section \ref{section:minkowski} we consider a related question of topological criteria for a convex compactum $A$ to be a Minkowski summand of a convex body $B$, giving a version of the results in \cite{goodey1982} and \cite[Section 6]{holmsenkar2017}.

\begin{remark}
The formula $\conv_K X = K\gdiv (K\gdiv X)$ may be interpreted so that when $X$ is not contained in a translate of $K$ then $\conv_K X = \mathbb R^n$, if $K$ is a convex body in $\mathbb R^n$. In Definition~\ref{definition:strong-convex-hull} we preferred to leave $\conv_K X$ undefined in this case. Assuming the definition $\conv_K X = \mathbb R^n$ for $X$ not contained in a translate of $X$, it is easy to investigate the Carath\'eodory number for such sets $X$, it equals the Helly number for the family of translates of $K$. This is a standard exercise which we leave to the reader.
\end{remark}

\subsection*{Acknowledgments}
The authors thank Rolf Schneider and the unknown referees for useful remarks and relevant references.

\section{Proof of Theorem \ref{theorem:caratheory-outside} about the Carath\'eodory number outside the ordinary convex hull}
\label{section:helly-type}

The proof mostly follows the argument in \cite{holmsenkar2017}. Assume the contrary, that the point, which we may assume to be the origin, does not belong to any $\conv_K Y$ for $Y\subseteq X$ with $|Y|\le n$, does not belong to the ordinary convex hull $\conv X$, but belongs to $\conv_K X$. Using the formulas 
\[
0\in K\gdiv (K\gdiv Y) \Leftrightarrow K\gdiv Y\subseteq K \Leftrightarrow \bigcap_{x\in Y} (K - x)\setminus K = \emptyset
\]
we conclude that 
\[
\bigcap_{x\in Y} (K - x)\setminus K \neq \emptyset,
\]
whenever $Y\subset X$ and $|Y|\le n$, while 
\[
\bigcap_{x\in X} (K - x)\setminus K = \emptyset.
\]

Let us introduce the system of sets 
\[
F_x = (K-x)\setminus K,\quad x\in X,
\]
in order to have a contradiction we need to prove that the intersection of the family of such sets is non-empty, provided every intersection of its subfamily of no more than $n$ sets is non-empty. Note that if we were not subtracting $K$ from the translates then the question we study would be just the question of the Helly number of the system of translates of $K$ (well studied in \cite[Chapter~III]{boltyanski1997}, for example).

The proof in \cite{kar2001car,holmsenkar2017} exploits the topological Helly theorem for the family of sets $\{F_x\}$ with Helly number $n+1$, which we now need to reduce to $n$ under the additional assumption $0\not\in\conv X$. In order to prove that the Helly number $n$ is sufficient, we assume, after an appropriate translation of $K$, that $0\in\inte K$, take a slightly inflated $K_\epsilon = (1+\epsilon) K$, and consider $G_x = (K - x)\cap \partial K_\epsilon$ instead of $F_x$ trying to show the following for sufficiently small $\epsilon>0$:

\begin{enumerate}
\item
Any $n$ or less of the sets $\{G_x\}$ have a common point;
\item
The sets $\{G_x\}$ do not cover the whole topological $(n-1)$-sphere $\partial K_\epsilon$;
\item
Any subfamily of the family $\{G_x\}$ has either empty or acyclic intersection.
\end{enumerate}

The \emph{acyclicity} here means vanishing of the reduced \v{C}ech cohomology of the set. The inclusion $u\in G_x$ means that $u \in \partial K_\epsilon$ and $u\in K-x$. We rewrite $0\in \partial K_\epsilon - u$ and $x\in K - u$ and note that the meaning of this is that $K_\epsilon$ translated by $-u$ has origin on its boundary and $K$ translated by $-u$ has $x$ inside.

Claim (1) can be established as follows. $F_{x_1}\cap \dots \cap F_{x_n}\neq \emptyset$ by the assumption of the theorem. This means that $Y = \{x_1,\ldots, x_n\}$ can be covered by a translate of $K-u$ not containing the origin. For sufficiently small $\epsilon$, the inflated $K_\epsilon-u$ does not contain the origin either. The theorem assumes that the whole $X$ can be covered by some $K-v$ and therefore by the inflated $K_\epsilon-v$. If $0\not\in K-v$ then this is what we need to prove; we proceed assuming the contrary, $0\in K-v$. Linearly interpolating between $K-u$ and $K-v$ with $K-(1-t)u - tv$, we find $K-w$ that contains $Y$ such that $0$ is on the boundary of $K_\epsilon-w$. Therefore $w\in \partial K_\epsilon$ and $w\in K - x_i$ for every $x_i\in Y$. In this argument we needed sufficiently small $\epsilon$ for every choice of $Y\subseteq X$ and eventually we are able to choose $\epsilon>0$ so small that it is suitable for all choices of $Y$.

Claim (2) follows from the assumption $0\not\in\conv X$. Taking a linear functional $\lambda$ such that $\lambda(X)>0$ we see that the point of $K_\epsilon$ with largest $\lambda$ cannot be covered by any $K-x$, $x\in X$.

Claim (3) needs the generating set property. First recall the useful fact from \cite{kar2001car} (in Russian) or \cite[Lemma 2.1]{holmsenkar2017} (in English with a simplified proof):
\begin{lemma}
\label{lemma:acyclic}
For two convex compacta $S, T\subset\mathbb R^n$ the set $S\setminus (S+T)$ is either empty or acyclic.
\end{lemma}

Now we explain Claim (3) using this lemma. The generating property of $K$ implies that any non-empty set $K\gdiv Y$ is a Minkowski summand of $K$, that is $K\gdiv Y + T = K$ for a convex compactum $T$. Hence $(K\gdiv Y)\setminus K$ is either empty or acyclic by the lemma. Let us show a similar fact, that any set of the form
\[
(K\gdiv Y) \setminus \inte K_\epsilon
\]
is either empty or acyclic. We note that $K\gdiv Y$ is a Minkowski summand of $(1+\delta)K$, just because for any $\delta>0$
\[
(K\gdiv Y) + T + \delta K = (1+\delta) K.
\]
By Lemma \ref{lemma:acyclic} we know that
\[
(K\gdiv Y) \setminus (1 + \delta) K
\]
is either empty or acyclic. If we make $\delta$ increase to $\epsilon$, we represent the compact set $(K\gdiv Y) \setminus \inte K_\epsilon$ as an intersection of the decreasing sequence of sets $(K\gdiv Y) \setminus (1 + \delta) K$ and use the continuity property of the \v{C}ech cohomology to ensure that the intersection is also empty or acyclic.

It remains to note that the set
\[
G_Y = (K\gdiv Y)\cap \partial K_\epsilon = \bigcap_{x\in Y} G_x
\]
is obtained as a central projection of $(K\gdiv Y)\setminus \inte K_\epsilon$ from a center in $(K\gdiv Y) \cap \inte K_\epsilon$ to the boundary $\partial K_\epsilon$ and any fiber of this projection is a segment (and is acyclic). Now we use another topological from \cite{begle1950,begle1955}:

\begin{lemma}[The Vietoris--Begle theorem]
\label{lemma:leray-map}
Let $f : X\to Y$ be a proper continuous map between metric spaces such that for every $y\in Y$, the fiber $f^{-1}(y)$ is acyclic. Then $f$ induces an isomorphism of the \v{C}ech cohomology of $X$ and $Y$.
\end{lemma}

Thus we conclude that $G_Y = \bigcap_{x\in Y} G_x$ is either empty or acyclic. We want to apply the topological Helly theorem: \emph{The family of compacta $\{G_x\}$ has non-empty intersection provided every union of its subfamily has vanishing cohomology in dimension $n-1$ and above and every intersection of its subfamily is either empty or acyclic.}

The claims (1), (2), (3) about the family $\{G_x\}$ verify the hypothesis of the topological Helly theorem. For (1) and (3) it is clear, while Claim (2) means that the union of this family or any its subfamily is contained in a topological sphere without one point, and hence has vanishing cohomology in dimension $n-1$ and higher.

The topological Helly theorem now implies that all the $G_x$ intersect in a point $u\in\partial K_\epsilon$. It means that the translate $K-u$ contains the whole $X$ and $K_\epsilon-u$ has $0$ on its boundary at the same time. Hence $K-u\subseteq \inte (K_\epsilon - u)$ does not contain the origin, verifying that the origin is not in $\conv_K X$. This contradiction completes the proof of the theorem.

\begin{remark}
We recommend to an interested reader the textbook \cite[\S 5]{godement1958} for further references on \v{C}ech cohomology (in which we mean the acyclicity), the covering resolution for the sheaf cohomology and other classical ideas that imply the topological Helly's theorem in the formulation that we use. A relevant recent paper about different versions of the topological Helly theorem is \cite{montejano2014}. 
\end{remark}

\begin{remark}
As it was noticed by one of the referees, our use of the topological Helly theorem may be easier to understand if we say that Claim (2) for the family $\{G_x\}$ means that this family misses a point of the topological sphere $\partial K_\epsilon$. Hence after a homeomorphism of this punctured sphere to $\mathbb R^{n-1}$ we obtain a more familiar situation of the topological Helly theorem in a Euclidean space.
\end{remark}

\begin{remark}
The case of $R$-strong convexity in this theorem, when $K$ is a Euclidean ball, may be done with the application of the ordinary Helly theorem for convex spherical caps not requiring any topology: A family of convex spherical caps not covering the whole sphere $\mathbb S^{n-1}$ has Helly number $n$. See also \cite[Theorem 5.7]{bezdek2007} for a similar argument.
\end{remark}

\section{Proof of Theorem \ref{theorem:product-gen} about the Carath\'eodory number when the gauge set is a product of generating sets}
\label{section:by-segment}

This theorem can be established with the techniques of the paper \cite{sierksma1975}, but we present a self-contained argument for reader's convenience. We start with a lemma that establishes a useful property of strong convexity.

\begin{lemma}
\label{lemma:exchange}
Let $K\subset\mathbb R^n$ be a generating set and let $X$ be a set of $n+2$ points. Then any point $p\in \conv_K X$ is contained in $\conv_K X'$ and $\conv_K X''$ for at least two distinct $X',X''\subset X$ of $n+1$ points each.
\end{lemma}

\begin{proof}
First note that the same thing is true for ordinary convexity: Any point $p\in \conv X$ is in $\conv X'$ and $\conv X''$ for at least two distinct $X',X''\subset X$ of at most $n+1$ points each. To show this we consider the linear map $f : \partial\Delta^{n+1}\to \mathbb R^n$, mapping the vertices of the simplex $\Delta^{n+1}$ to the points of $X$. This map has degree zero on the boundary of the simplex and assuming the contrary, $p$ lying in the image of only one facet of the simplex, we see that the degree cannot be zero, a contradiction.

Now we proceed to the case of $p\in \conv_K X$. If $p$ is in the ordinary convex hull then it is in the ordinary convex hulls of at least two distinct proper subsets of $X$, which in turn lie in the $K$-strongly convex hulls, and we are done. Otherwise by Theorem \ref{theorem:caratheory-outside} $p$ is in a convex hull of some $X'''\subset X$ of at most $n$ points, and we can choose two distinct $X',X''\supset X'''$ of $n+1$ points each.
\end{proof}

Now we prove the theorem. Let $K=L\times M$, take $X\subset\mathbb R^n$, and $p\in \conv_K X$. Denote by $P$ and $Q$ the projections of $\mathbb R^n$ onto the complementary $\mathbb R^\ell$ and $\mathbb R^m$ respectively, corresponding to the decomposition $L\times M$. Evidently, 
\[
\conv_K X = \conv_L P(X)\times \conv_M Q(X).
\]
By the Carath\'eodory theorem for $L$-strong convexity we find a subset $Y\subseteq X$ of at most $\ell+1$ points so that $P(p)\in \conv_L P(Y)$; by the Carath\'eodory theorem for $M$-strong convexity we find a subset $Z\subseteq X$ of at most $m+1$ points so that $Q(p)\in \conv_M P(Z)$.

In total, we have at most $\ell + m + 2 = n+ 2$ points in $Y\cup Z$ and $p\in \conv_K(Y\cup Z)$. 


The case $|Y\cup Z|\le n$ is trivial.

Consider the case $|Y\cup Z| = n+1$. In this case either $Y$ has $\ell+1$ points or $Z$ has $m+1$ points, since otherwise $|Y\cup Z| \le |Y| + |Z| \le \ell+m=n$. Without loss of generality, let $Z$ have $m+1$ points.

If we take any point $y \in Y$ then by Lemma \ref{lemma:exchange} applied to $M$, $Q(p)$, and $Q(Z\cup\{y\})$, there is a corresponding point $z \in Z$ such that 
\begin{equation}
\label{equation:drop}
Q(p) \in \conv_M Q\left( Z \cup \{y\} \setminus \{z\}\right).
\end{equation}
It follows that $p$ is in the $K$-strongly convex hull of $(Z\cup\{y\}\setminus\{z\})\cup Y$, which has at most $n$ points.

It remains to consider the case $|Y\cup Z| = n+2$, this means that $Y$ and $Z$ are disjoint and have $\ell+1$ and $m+1$ points respectively. What we want to do is to drop one point $y$ from $Y$ and one point $z$ from $Z$ so that the remaining $n$ points of 
\[
X' = (Y\setminus \{y\})\cup (Z\setminus\{z\})
\]
still possess the property that $P(p)\in \conv_L P(X')$ and $Q(p)\in \conv_M Q(X')$, implying $p\in \conv_K(X')$

Choose a map $f : Y \to Z$ taking any $y$ to a $z$ as in \eqref{equation:drop}. In the other direction, as we now have $|Y|=m+1$, consider the corresponding map $g$ (arising from the projection $P$ and the $L$-strong convexity) from $Z$ to $Y$. 

It is sufficient to find $y^* \in Y$ and $z^* \in Z$ such that $f(y^*) \neq z^*$ and $g(z^*) \neq y^*$, dropping $f(y^*)$ and $g(z^*)$ we will keep $P(p)$ contained in $\conv_L P(X')$ and $Q(p)$ contained in $\conv_M Q(X')$ from the choice of $f$ and $g$. We just note that $f$ and $g$ make edges of a bipartite graph on $Y\sqcup Z$, and when
\begin{equation}
\label{equation:big}
(\ell + 1)(m + 1) > \ell + m + 2,
\end{equation}
we can take $y^*\in Y$ and $z^*\in Z$ not connected by an edge in any direction. 

Inequality \eqref{equation:big} only fails for $\ell = m = 1$. In this case $K$ is a square in the plane and we can argue differently. Put $p$ to the origin and then it remains to find two points in $X$ that lie in opposite quadrants relative to $p$ to have $p$ in their $K$-strongly convex hull. But such a situation means that a vertical or a horizontal line separates $p$ from $X$, which evidently contradicts the inclusion $p\in \conv_K(X)$.


We have proved that the Carath\'eodory number is at most $n$, it is actually precisely $n$ because of Theorem~ \ref{theorem:at-least-n}.

\section{Proof of Theorem \ref{theorem:at-least-n}, that the Carath\'eodory number for strong convexity is not less than the dimension}

We assume $n\ge 2$, because for $n=1$ all convex bodies $K\subset \mathbb R$ are segments and the Carath\'eodory number is always $2$.

Consider the John ellipsoid of $K$, that is the ellipsoid of maximal volume contained in $K$. If we make an affine transformation to make the ellipsoid a unit ball $B$ centered at the origin then (see \cite{ball1997}) the set of tangency points $B\cap \partial K$ linearly spans the whole $\mathbb R^n$. Hence it contains certain points $v_1,\ldots,v_n$ that produce a basis of $\mathbb R^n$, then the normals to both $B$ and $K$ at those points are the same vectors $v_1,\ldots, v_n$.

Now make a linear transformation so that the points of $\partial K$ become some $p_1,\ldots, p_n$, while the outer normals to $K$ at those points become an orthonormal basis $e_1,\ldots, e_n$. Since there is an ellipsoid inner tangent to $K$ at every $p_i$, we conclude that there is a ball of some radius inner tangent to $K$ at every point $p_i$. Let us scale $K$ until those balls have radius greater or equal to $R$, satisfying the inequalities
\begin{equation}
\label{equation:big-radius}
R^2 \ge 1 + \left(R - \frac{1}{2n}\right)^2\Leftrightarrow \frac{R}{n} \ge 1  + \frac{1}{4n^2} \Leftrightarrow R \ge n + \frac{1}{4n}.
\end{equation}

Now let $X$ be just the set $\{e_1,\ldots,e_n\}$ and 
\[
p = \frac{e_1 + \dots + e_n}{n}.
\]
Evidently $p\in \conv X \subseteq \conv_K X$ and in order to prove the theorem it remains to show that $p$ is not contained in the $K$-strongly convex hull of a proper subset $Y\subset X$. Without loss of generality assume $Y=\{e_2,\ldots,e_n\}$.

The hyperplane $H = \{x : x\cdot e_1 = 1/(2n)\}$ separates $p$ from $Y$. Translate $K$ by the vector 
\[
\frac{1}{2n} e_1 - p_1,
\]
after this translation $H$ is a tangent plane to $K$ at $e_1/(2n)$ and therefore the translated $K$ does not contain $p$. Also, after the translation $K$ contains the ball $B_1$ of radius $R$ centered at 
\[
\left(\frac{1}{2n} - R\right) e_1.
\]
Assumption \eqref{equation:big-radius} implies that $B_1$ also contains $e_2,\ldots, e_n$. Hence the translate of $K$ we consider contains $Y$ and does not contain $p$, proving that $p\not\in \conv_K Y$.

\section{Yet another example with an infinite Carath\'eodory number}
\label{section:infinite}

Let us give an example of an infinite Carath\'eodory number, which looks simpler and more geometric compared to the examples in \cite[Theorem 8]{langi2013} and \cite[Section 5]{holmsenkar2017}. 

Let $K$ be a cone in $\mathbb R^3$ with a unit disk base in the plane $A$ and the apex above $A$. Let $X$ be the vertices of a regular polygon inscribed in a circle of radius $1/2$ in the plane $A$. Assume $|X|\ge 4$, since otherwise some of the pictures get too degenerate.

We will show that there exists a point $p\in\conv_K X$, which is not in any $\conv_K X'$ for any proper subset $X'\subset X$. 

Put $H = \left( \conv_K X\right) \cap A.$ This planar set $H$ has boundary consisting of convex arcs between consecutive points of $X$ of radius equal one. This is so because the intersection of a translate $K+t$ with the plane $A$ is a disk of radius at most $1$ and it must contain thus described curvilinear regular polygon $H$.

Let us take a point $p$ to be in $\mathbb R^3$ above the center of $X$, that is on the line through the center of $X$ perpendicular to $A$, such that the maximum of the angle $\angle (py, A)$ over $y\in \partial H$ equals the angle $\alpha$ between the base of $K$ and any generatrix of $K$. This maximum is attained in the midpoint of every curved segment of the boundary $\partial H$. Our choice of $p$ guarantees that any translate of $K$ containing $X$ (whose intersection with $A$ is a disk of radius at most $1$) must contain the point $p$ as well. This is so, since otherwise a plane $B$ with $\angle (A,B)=\alpha$ would separate $H$ from $p$ and touch $H$, which we exclude by the choice of $p$.

If we drop a point to have $X'=X\setminus \{x\}$ instead of $X$, then we have $H'$ instead of $H$, with one arc of $\partial H'$ longer than others and having the removed point $x$ strictly outside, see Figure~\ref{figure:cone}. Consider a translate of $K$ with base in $A$ such that the boundary of $K\cap A$ passes through this new arc (its base is the big circle in the picture). We observe that $p$ is not in this translate of $K$. This is because the new arc is closer to the center of $X$ compared to $\partial H$ and the maximum of the angle $\angle (py, A)$ over $y\in \partial H'$ is now attained in the center of this new arc, and is greater than $\alpha$.

Hence $p$ is not in $\conv_K X'$ for any proper subset $X'\subset X$ and the Carath\'eodory number for $K$-strong convexity is at least $|X|$, which can be arbitrarily high in our construction with the fixed cone $K$.

\begin{figure}
\includegraphics[scale=1.0]{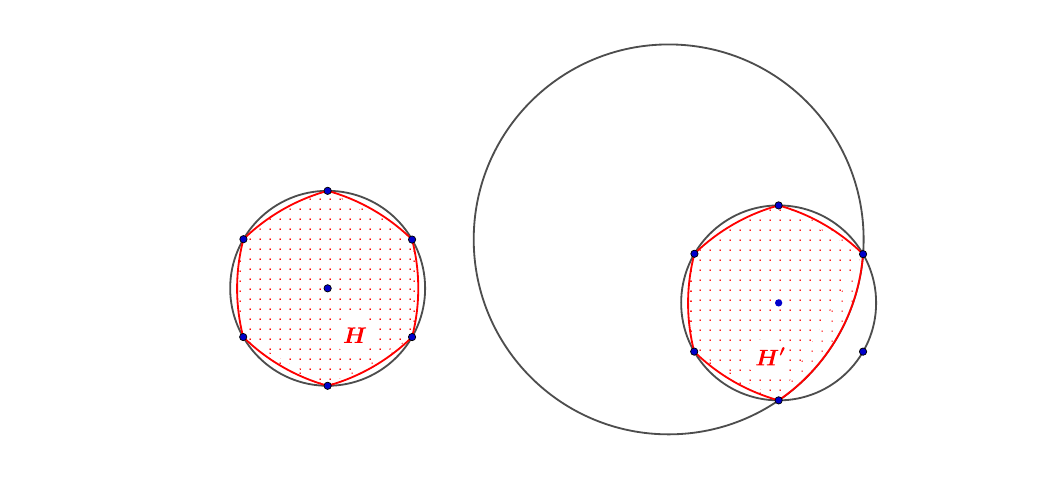}
\caption{The point $p$ is above the blue center of the circle.}
\label{figure:cone}
\end{figure}

\section{Another criterion for the Minkowski summand}
\label{section:minkowski}

Developing the property of acyclicity of $A\setminus (A + B)$ used in \cite[Lemma 2.1]{holmsenkar2017} we referenced above, in \cite{holmsenkar2017} it was shown that given a convex body $A\subset \mathbb R^n$ and a convex open bounded set in $B\subset \mathbb R^n$ such that for any vector $t\in\mathbb R^n$, the set $(A+t)\setminus B$ is either empty or acyclic, the set $A$ is a Minkowski summand of $B$. The latter means there exists an open convex $C$ such that $B = A + C$. 

In \cite{goodey1982} a similar criterion was established, essentially the opposite of what we used in Section~\ref{section:helly-type}: If $(A+t)\cap \partial B$ is always acyclic for every translation $t$ and fixed convex bodies $A,B\subset\mathbb R^n$, then $A$ is a Minkowski summand of $B$. 

Here we establish another criterion for a Minkowski summand, using essentially the same technique as in \cite{goodey1982,holmsenkar2017}. Again, the acyclicity is assumed in terms of the \v{C}ech cohomology.

\begin{theorem}
Let $A$ be a convex compactum in $\mathbb R^n$ and $B$ be a convex body in $\mathbb R^n$. Assume that a translate $A+t_0$ is contained in the interior of $B$, and for any vector $t\in\mathbb R^n$, if $(A+t)$ is contained in $B$ then the set $(A+t)\cap \partial B$ is either empty or acyclic. Then $A$ is a Minkowski summand of $B$.
\end{theorem}

\begin{proof}
Put $C = B\gdiv A$, and note that $C$ is a convex body since it contains $t_0$ in its interior, we will show that $A+C = B$. For any $t\in \partial C$ we have $A + t \subseteq B$, hence $(A+t)\cap \partial B$ is an acyclic compactum, it cannot be empty since otherwise $A+t\subseteq \inte B$ and from the positive distance between $A+t$ and $\mathbb R^n\setminus \inte B$ there would exist a neighborhood $U(t)$ such that $A+U(t)\subseteq B$, which would imply $t\in \inte C$, not in $\partial C$.

Now consider the set
\[
Z = \left\{ (x,t) : x\in \partial B,\ t\in \partial C,\ x \in A + t \right\},
\]
it is compact and its projection to $\partial C$ has compact acyclic fibers $(A+t)\cap \partial B$ as already shown. Hence by the Vietoris--Begle theorem (Lemma \ref{lemma:leray-map} above), $Z$ has the same \v{C}ech cohomology as $\partial C$, which is the $(n-1)$-sphere.

Consider now the projection of $Z$ to the first summand, $\partial B$. A fiber of this projection over $x\in \partial B$ is
\[
\left\{ t \in\partial C: A + t \ni x \right\} = \left\{ t \in C: A + t \ni x \right\} = C \cap (t - A),
\]
which is a, possibly empty, convex compactum. Let $S\subset \partial B$ be the image of this projection, the Vietoris--Begle theorem implies that the \v{C}ech cohomology of $S$ is the same as that of $Z$ and therefore its cohomology is the cohomology of the $(n-1)$-sphere. If $S$ is not the whole $\partial B$ then its $(n-1)$-dimensional cohomology would vanish, since it can be calculated as the limit of the cohomologies of its neighborhoods, open $(n-1)$-manifolds. 

Hence $S$ must be the whole $\partial B$ and then $A+ C \supseteq \partial B$, hence $A+C = B$.
\end{proof}

\bibliography{../Bib/karasev}
\bibliographystyle{abbrv}
\end{document}